\documentclass[11pt]{article}

\usepackage{amsfonts}
\usepackage{amssymb,amsmath}
\usepackage{amsthm}
\usepackage{latexsym}
\usepackage{graphics}
\usepackage{epic}
\usepackage{epsfig}
\usepackage{psfrag}
\usepackage{bbm}
\usepackage{dsfont}
\usepackage{enumitem}
\usepackage{stmaryrd}
\usepackage{xcolor}
\usepackage{hyperref}

\hypersetup{colorlinks=true,citecolor={red}}

\numberwithin{equation}{section}
\def\PP{\mathbb{P}}
\def\QQ{\mathbb{Q}}
\def\RR{\mathbb{R}}

\def\EE{\mathbb{E}}
\def\11{\mathbbm{1}}

\def\E{\mathbb{E}}
\def\P{\mathbb{P}}
\def\R{\mathbb{R}}

\def\N{\mathbb{N}}

\def\d{\partial}

\def\tT{\widetilde{T}}
\def\tX{\widetilde{X}}

\newtheorem{thm}{Theorem}[section]
\newtheorem{lem}[thm]{Lemma}
\newtheorem{cor}[thm]{Corollary}

\newtheorem{prop}[thm]{Proposition}

\theoremstyle{remark}
\newtheorem{rem}{Remark}

    \def\restriction#1#2{\mathchoice
                  {\setbox1\hbox{${\displaystyle #1}_{\scriptstyle #2}$}
                  \restrictionaux{#1}{#2}}
                  {\setbox1\hbox{${\textstyle #1}_{\scriptstyle #2}$}
                  \restrictionaux{#1}{#2}}
                  {\setbox1\hbox{${\scriptstyle #1}_{\scriptscriptstyle #2}$}
                  \restrictionaux{#1}{#2}}
                  {\setbox1\hbox{${\scriptscriptstyle #1}_{\scriptscriptstyle #2}$}
                  \restrictionaux{#1}{#2}}}
    \def\restrictionaux#1#2{{#1\,\smash{\vrule height .8\ht1 depth .85\dp1}}_{\,#2}}

\begin{document}

\title{Uniform convergence of conditional distributions for absorbed one-dimensional diffusions}

\author{Nicolas Champagnat$^{1,2,3}$, Denis Villemonais$^{1,2,3}$}

\footnotetext[1]{Universit\'e de Lorraine, IECL, UMR 7502, Campus Scientifique, B.P. 70239,
  Vand{\oe}uvre-l\`es-Nancy Cedex, F-54506, France}
\footnotetext[2]{CNRS, IECL, UMR 7502,
  Vand{\oe}uvre-l\`es-Nancy, F-54506, France} \footnotetext[3]{Inria, TOSCA team,
  Villers-l\`es-Nancy, F-54600, France.\\
  E-mail: Nicolas.Champagnat@inria.fr, Denis.Villemonais@univ-lorraine.fr}

\maketitle

\begin{abstract}
  This article studies the quasi-stationary behaviour of absorbed one-dimensional diffusions. We obtain necessary and sufficient
  conditions for the exponential convergence to a unique quasi-stationary distribution in total variation, uniformly with respect to
  the initial distribution. An important tool is provided by one dimensional strict local martingale diffusions coming down from
  infinity. We prove under mild assumptions that their expectation at any positive time is uniformly bounded with respect to the
  initial position. We provide several examples and extensions, including the sticky Brownian motion and some one-dimensional
  processes with jumps.
\end{abstract}



\section{Introduction}
\label{sec:intro}

This article studies the quasi-stationary behaviour of general one-di\-men\-sio\-nal diffusion processes in an interval $E$ of $\RR$,
absorbed at its finite boundaries. When the process is absorbed, it is sent to some cemetary point $\d$. We recall that a
\emph{quasi-stationary distribution} for a continuous-time Mar\-kov process $(X_t,t\geq 0)$ on the state space $E\cup\{\d\}$ absorbed
in $\d$, is a probability measure $\alpha$ on $E$ such that
$$
\PP_\alpha(X_t\in\cdot\mid t<\tau_\partial)=\alpha(\cdot),\quad \forall t\geq 0,
$$
where $\PP_\alpha$ denotes the distribution of the process $X$ given that $X_0$ has distribution $\alpha$, and
$$
\tau_\d:=\inf\{t\geq 0: X_t=\d\}.
$$
We refer to~\cite{meleard-villemonais-12,vanDoorn2013,pollett-11} for general introductions to the topic.

We consider a general diffusion $(X_t)_{t\geq 0}$ on $[a,b)$ with $a<b\leq +\infty$ absorbed at $\d=a$ in a.s.\ finite time $\tau_\d$
and such that $b$ is an entrance or natural boundary. We are interested in the existence of a quasi-stationary distribution $\alpha$
on $(a,b)$ such that, for all probability measures $\mu$ on $(a,b)$,
\begin{align}
\label{eq:QLD}
\left\|\PP_\mu(X_t\in \cdot\mid t<\tau_\partial)-\alpha\right\|_{TV}\leq C e^{-\gamma t},\ \forall t\geq 0
\end{align}
where $\|\cdot\|_{TV}$ is the total variation norm and $C$ and $\gamma$ are positive constants. We prove that this property is
\emph{equivalent} to the following condition:
\begin{description}
\item[\textmd{(B)}] The boundary $b$ is entrance for $X$ and there exists $t,A>0$ such that
  \begin{equation*}
    \P_x(\tau_\d> t)\leq As(x),\quad\forall x>0,
  \end{equation*}
  where $s$ is the scale function of the diffusion $X$.
\end{description}
In the sequel, without loss of generality, we will focus on the case of a diffusion on natural scale on $E\cup\{\d\}=[0,+\infty)$ and modify
(B) accordingly.

The exponential convergence~\eqref{eq:QLD} entails number of additional properties. First, it implies the uniqueness of the
quasi-stationary distribution $\alpha$ which attracts any initial distribution. Second,~\eqref{eq:QLD} implies
~\cite[Theorem~2.1]{ChampagnatVillemonais2016U} that the following convergence is uniform in $x\in(a,b)$:
\begin{align}
  \label{eq:convergence-to-eta}
  \eta(x)=\lim_{t\rightarrow\infty} \frac{\P_x(t<\tau_\d)}{\P_\alpha(t<\tau_\d)}=\lim_{t\rightarrow +\infty} e^{\lambda_0 t}\P_x(t<\tau_\d).
\end{align}
Note that we used the existence of a constant $\lambda_0>0$ such that
$$
\PP_\alpha(t<\tau_\d)=e^{-\lambda_0 t},\quad\forall t\geq 0,
$$
which is a well-known property of quasi-stationnary distributions. Moreover~\cite[Prop.\,2.3]{ChampagnatVillemonais2016}, $\eta$
belongs to the domain of the infinitesimal generator $L$ of $X$ on the Banach space of bounded measurable functions on $[a,b)$
equipped with the $L^\infty$ norm and
\begin{align*}
  L\eta=-\lambda_0\eta.
\end{align*}
Third, it implies the existence and the exponential ergodicity of the associated $Q$-process, defined as the process $X$ conditionned
to never be extinct (see~\cite[Thm.\,3.1]{ChampagnatVillemonais2016} for a precise definition). The convergence of the conditional
laws of $X$ to the law of the $Q$-process holds also uniformly in total variation
norm~\cite[Theorem~2.1]{ChampagnatVillemonais2016U}, which also entails conditional ergodic
properties~\cite[Corollary~2.2]{ChampagnatVillemonais2016U}.

The second main result of this paper is the fact that Condition~(B) is closely related to properties of strict local martingale
diffusions $(Z_t,t\geq 0)$ on $[0,\infty)$. We recall that $(Z_t,t\geq 0)$ is called a strict local martingale if $\EE_x(Z_t)<x$ for all $x,t>0$. We say that $(Z_t,t\geq 0)$ is a \emph{strict martingale in the strong sense} if, for all
$t>0$,
\begin{equation}
  \label{eq:strong-strict}
  \sup_{x>0}\EE_x(Z_t)<\infty.
\end{equation}
We prove, under a mild control on the oscillations of the speed measure near $+\infty$, that the property~\eqref{eq:strong-strict} is true for all $t>0$. This new result has its own interest in
the theory of strict local martingales~\cite{Protter2013} and applies to nearly all strict local martingale diffusions on
$[0,\infty)$. For instance, it is well-known that the solution on $(0,+\infty)$ to the
SDE
\begin{align*}
  dZ_t=Z_t^\alpha dB_t
\end{align*}
is a strict local martingale if and only if $\alpha>1$. Our result implies that, for
all $\alpha>1$ and all $t>0$,~\eqref{eq:strong-strict} holds true, and hence that $\EE_x(Z_t)<x$ for all $x>0$ is actually
equivalent to $\sup_{x>0}\EE_x(Z_t)<\infty$. This equivalence is also true for example for the SDE
\begin{align*}
  dZ_t=Z_t\ln(Z_t)^\beta dB_t
\end{align*}
which is a (strong) strict local martingale iff $\beta>1/2$.

We deduce from these results an explicit criterion ensuring Condition~(B), which shows that it is satisfied for nearly all diffusions
coming down from infinity and hitting 0 in a.s.\ finite time. However, in several cases, Condition~(B) is simple enough to be checked using elementary
probability tools, and to apply to many other settings (processes with jumps, piecewise deterministic
paths~\cite{ChampagnatVillemonais2016} or time-inhomogeneous processes~\cite{champagnat-villemonais-16c}). Hence we also provide a
simple probabilistic criterion and show how it applies in an example of one-dimensional process with jumps.

We also provide an explicit formula relating the speed measure of a diffusion on natural scale and its quasi-stationary distribution.
Conversely, we show that, given a probability measure $\alpha$ on $(0,\infty)$ satisfying suitable conditions and a positive constant
$\lambda_0$, there exists a unique diffusion on natural scale, with explicit speed measure, admitting $\alpha$ as unique
quasi-stationary distribution with associated absorbing rate $\lambda_0$. We also use this explicit formula to compute the
quasi-stationary distribution of the sticky Brownian motion on $(-1,1)$, absorbed at $\{-1,1\}$.
\bigskip

The usual tools to prove convergence as in~\eqref{eq:QLD} for conservative Markov processes involve coupling arguments (Doeblin's
condition, Dobrushin coupling constant, see e.g.~\cite{Meyn2009}). However, a diffusion conditioned not to hit 0 before a given time
$t>0$ is a time-inhomogeneous diffusion process with a singular, non-explicit drift for which these methods fail. For instance, the
solution $(X_s)_{s\geq 0}$ to the SDE
\begin{align*}
  dX_s=\sigma(X_s)dB_s+b(X_s)ds
\end{align*}
with smooth coefficients and conditioned not to hit 0 up to a time $t>0$ has the law of the solution $(X^{(t)}_s)_{s\in[0,t]}$ to the
SDE
\begin{align*}
dX^{(t)}_s=b(X^{(t)}_s)ds+\sigma(X^{(t)}_s)dB_s+\sigma(X^{(t)}_s)\,\left[\nabla \ln \P_{\cdot}(t-s<\tau_\d)\right](X^{(t)}_s)ds.
\end{align*}
Since $\P_{x}(t-s<\tau_\d)$ vanishes when $x$ converges to 0, the drift term in the above SDE is singular and cannot in general be
quantified.

Hence, convergence of conditioned diffusion processes have been obtained up to now using (sometimes involved) spectral theoretic
arguments, which most often require regularity of the coefficients (see for
instance~\cite{Cattiaux2009,littin-12,kolb-steinsaltz-12,kolb-hening-2104,miura-14}) and are specific to time homogeneous
one-dimensional diffusion processes. None of the above references provide uniform convergence with respect to the initial
distribution. In particular, all the consequences of~\eqref{eq:QLD} listed above are original results of this paper, even for
solutions to SDEs with regular coefficients. Moreover, the uniformity of the convergence with respect to the initial distribution is
important in applications, as discussed in~\cite[Ex.\,2]{meleard-villemonais-12}. Concerning the sole existence and uniqueness of a
quasi-stationnary distribution, in the few cases where our criterion is not satisfied, we refer the reader
to~\cite{littin-12,kolb-steinsaltz-12,kolb-hening-2104}. An interesting open problem is to prove that~\eqref{eq:QLD}, and hence
Condition~(B), actually holds true in these cases.

The paper is organized as follows. In Section~\ref{sec:construction}, we precisely define the absorbed diffusion processes and recall
their construction as time-changed Brownian motions. Section~\ref{sec:pties-diff} contains several results on strict local
martingales, including the criterion ensuring~\eqref{eq:strong-strict}. We give in Section~\ref{sec:1d_diffusions} the statements of
all our main results on quasi-stationnary distributions (Subsections~\ref{sec:QSD-without-killing} and~\ref{sec:dalpha-dm}) and on
criteria ensuring Condition~(B) (Subsections~\ref{sec:condition-(B)} and~\ref{sec:strict-local-martingales}). We also study in
Subsection~\ref{sec:examples} several examples including sticky Brownian motion and diffusions with jumps. Finally, we give in
Section~\ref{sec:proof-QSD-diff} the proofs of the main results of Section~\ref{sec:1d_diffusions}.

\section{Some reminders on diffusion processes}
\label{sec:construction}

In this section, we recall the construction of general diffusion processes $(X_t,t\geq 0)$ on $[0,+\infty)$, absorbed at
$\partial=0$. The typical situation corresponds to stochastic population dynamics of continuous densities with possible extinction.
 
The distribution of $X$ given $X_0=x\in[0,\infty)$ will be denoted $\PP_x$, and the semigroup of the process is given by
$P_tf(x)=\EE_x[f(X_t)]$ for all bounded measurable $f:[0,\infty)\rightarrow\RR$ and all $x\geq 0$.

A stochastic process $(X_t,t\geq 0)$ on $[0,+\infty)$ is called a diffusion if it has a.s.\ continuous paths in $[0,\infty)$,
satisfies the strong Markov property and is \emph{regular}. By regular, we mean that for all $x\in(0,\infty)$ and $y\in[0,\infty)$,
$\P_x(T_y<\infty)>0$, where $T_y$ is the first hitting time of $y$ by the process $X$. Examples of such processes are given by (weak)
solutions of stochastic differential equations (see Section~\ref{sec:examples}).

Given such a process, there exists a continuous and strictly increasing function $s$ on $[0,\infty)$, called the \emph{scale
  function}, such that $(s(X_{t\wedge T_0}),t\geq 0)$ is a local martingale~\cite{freedman-83}. The stochastic process $(s(X_t),t\geq
0)$ is itself a diffusion process with identity scale function. Since we shall assume that 0 is regular or exit and that $T_0<\infty$
a.s., we necessarily have $s(0)>-\infty$ and $s(\infty)=\infty$, and we can assume $s(0)=0$. Hence, replacing $(X_t,t\geq 0)$ by
$(s(X_t),t\geq 0)$, we can assume without loss of generality that $s(x)=x$. Such a process is said to be on \emph{natural scale} and
satisfies for all $0<a<b<\infty$,
$$
\P_x(T_a<T_b)=\frac{b-x}{b-a}.
$$
To such a process $X$, one can associate a unique locally finite positive measure $m(dx)$ on $(0,+\infty)$, called the \emph{speed
  measure} of $X$, which gives positive mass to any open subset of $(0,+\infty)$ and such that $X_t=B_{\sigma_t}$ for all $t\geq 0$
for some standard Brownian motion $B$, where
\begin{align}
  \label{eq:As}
  \sigma_t=\inf\left\{s>0:A_s>t\right\}, \text{ with }A_s=\int_0^\infty L^x_s\, m(dx).
\end{align}
and $L^x$ is the local time of $B$ at level $x$. Conversely, any such time change of a Brownian motion defines a regular diffusion on
$[0,\infty)$~\cite[Thm.\,23.9]{kallenberg-02}.

\medskip
We will use in the sequel the assumption that
\begin{align}
  \label{eq:ymdy_integrable}
  \int_0^\infty y\,m(dy)<\infty.
\end{align}
Since the measure $m$ is locally finite on $(0,\infty)$, this assumption reduces to local integrability of $y\,m(dy)$ near 0 and
$+\infty$. These two conditions are respectively equivalent to $\P_x(\tau_\d<\infty)=1$ for all
$x\in(0,\infty)$~\cite[Thm.~23.12]{kallenberg-02}, where $\tau_\d:=T_0$ is the first hitting time of $\d=0$ by the process $X$, and to
the existence of $t>0$ and $y>0$ such that
$$
\inf_{x>y}\P_x(T_y<t)>0.
$$
This means that $\infty$ is an \emph{entrance boundary} for $X$~\cite[Thm.\,23.12]{kallenberg-02}, or equivalently that the process
$X$ \emph{comes down from infinity}~\cite{Cattiaux2009}.

We recall that $0$ is absorbing for the process $X$ iff $\int_0^1\,m(dx)=\infty$~\cite[Thm.\,23.12]{kallenberg-02}, which is not
necessarily the case under the assumption that $\int_0^1 y\,m(dy)<\infty$. Therefore, we modify the definition of $X$ so that $\d$
becomes an absorbing point as follows:
\begin{equation}
  \label{eq:def-diff-no-kill}
  X_t=
  \begin{cases}
    B_{\sigma_t} & \text{if\ }0\leq t<\tau_\d \\
    \d & \text{if\ }t\geq\tau_\d.
  \end{cases}
\end{equation}
Note that this could be done equivalently by replacing $m$ by $m+\infty\delta_0$.

With this definition, $X$ is a local martingale.
Note that $A_s$ is strictly increasing since $m$ gives positive mass to any open interval, and hence $\sigma_t$ is continuous and
\begin{equation}
  \label{eq:lien-tau-d-brownien}
  \sigma_{\tau_\d}=T^B_0,  \quad\text{or, equivalently,}\quad \tau_\d=A_{T^B_0},
\end{equation}
where $T^B_x$ is the first hitting time of $x\in\RR$ by the process $(B_t,t\geq 0)$. In particular, $X_t=B_{\sigma_t\wedge T_0^B}$, $\forall t\geq 0$.


\section{Strict local martingales}
\label{sec:pties-diff}

This section is devoted to the study of one-dimensional strict local martingale diffusions on $[0,\infty)$ coming down from infinity
and their link with absorbed one-dimensional diffusion processes. We recall that a nonnegative, real local martingale $(Z_t,t\geq 0)$
with $Z_0=z\geq 0$ is called strict if it is not a martingale, i.e.\ if $\E(Z_t)<z$ for at least one $z\in (0,\infty)$. If
$(Z_t,t\geq 0)$ is a diffusion on $[0,\infty)$ on natural scale with speed measure $\widetilde{m}$, then
\begin{equation}
  \label{eq:str-loc-mg}
  (Z_t,t\geq 0)\mbox{ is a strict local martingale }\Longleftrightarrow\ \int_1^\infty y \,\widetilde{m}(dy)<\infty,  
\end{equation}
i.e.\ if $+\infty$ is an entrance boundary~\cite{kitani-06}. The notion of strict local martingales is of great importance in the
theory of financial bubbles. We refer the interested reader to Protter~\cite{Protter2013} and references therein.

The goal of this section is to study a stronger notion of strict martingale. We will say that a nonnegative, real local martingale
$(Z_t,t\geq 0)$ is a \emph{strict martingale in the strong sense} if, for all $t>0$, there exists $A_t<\infty$ such that
$$
\E_z(Z_t)\leq A_t,\ \forall z>0.
$$

\subsection{Strict local martingales in a strong sense}
\label{sec:local-martingale}

The next result gives a sufficient criterion for a strict local martingale diffusion process $Z$ on $[0,\infty)$ to be a strictly
local martingale in the strong sense.

\begin{thm}
  \label{thm:h-transfo-3-v2}
  Let $\widetilde{m}$ be a locally finite measure on $(0,\infty)$ giving positive mass to any open subset of $(0,\infty)$, and let
  $Z$ be a diffusion process on $[0,\infty)$ stopped at 0 on natural scale with speed measure $\widetilde{m}$. Assume
  \begin{equation}
    \label{eq:matsumoto-bis}
    \int_1^\infty\frac{1}{x}\,\sup_{y\geq x}\left(y\int_y^\infty \widetilde{m}(dz)\right)\,dx<\infty.
  \end{equation}
  Then, $(Z_t)_{t\geq 0}$ is a strict local martingale in the strong sense.
\end{thm}

Note that Fubini's theorem implies that
\begin{align*}
  \int_1^\infty (z-1)\,\widetilde{m}(dz) & = \int_1^\infty\int_y^{\infty}\widetilde{m}(dz)\,dy \leq\int_1^\infty\frac{1}{x}\left(\sup_{y\geq x}y\int_y^\infty\widetilde{m}(dz)\right)dx.
\end{align*}
Hence~\eqref{eq:matsumoto-bis} and the fact that $\widetilde{m}$ is locally finite implies that
\begin{equation}
  \label{eq:matsumoto-1}
  \int_1^\infty y\,\widetilde{m}(dy)<\infty.
\end{equation}

\begin{rem}
  \label{rem:matsumoto}
  Although we do not know if all strict local martingale diffusions are strict in the strong sense, the last result shows that it is
  true for most strict local martingale diffusions. Indeed, when $\widetilde{m}$ has a density with respect to Lebesgue's measure
  $\Lambda$ on $(0,\infty)$, we recall that $(Z_t,t\geq 0)$ is (weak) solution to the SDE
  $$
  dZ_t=\sigma(Z_t)dB_t,
  $$
  where $(B_t,t\geq 0)$ is a standard Brownian motion and $\widetilde{m}(dx)=\sigma^{-2}(x)\Lambda(dx)$. In this case, the
  condition~\eqref{eq:matsumoto-bis} covers almost all the practical situations where~\eqref{eq:matsumoto-1} is satisfied. For
  example, it is true if
  \begin{align*}
    \frac{d\widetilde{m}}{d\Lambda}(x)\leq \frac{C}{x^2\log x\cdots \log_{k-1} x\left(\log_k x\right)^{1+\epsilon}},
  \end{align*}
  or equivalently
  \begin{align*}
    \sigma(x)\geq C'x\sqrt{\log x\cdots \log_{k-1} x\left(\log_k x\right)^{1+\epsilon}},
  \end{align*}
  for all $x$ large enough and for some $\epsilon>0$ and $k\geq 1$, where $\log_k(x)=\log_{k-1}( \log x)$
  (see~\cite[Ex.\,2]{matsumoto-89}). In other words,~\eqref{eq:matsumoto-bis} could fail only when $\widetilde{m}$ has strong
  oscillations close to $+\infty$ and $y\,\widetilde{m}(dy)$ is nearly non-integrable at infinity.
\end{rem}

\begin{proof}
  Since $(Z_t,t\geq 0)$ is a positive local martingale, hence a supermartingale, and a strong Markov process, we have for all $z\geq 1$ and $t>0$,
  \begin{align*}
    \EE_z(Z_t) & =\EE_z(Z_t\mathbbm{1}_{t<T^Z_1})+\EE_z(Z_t\mathbbm{1}_{T^Z_1<t}) \\  
    & \leq\EE_z(Z_{t\wedge T^Z_1}\mathbbm{1}_{t<T^Z_1})+\EE_z[\mathbbm{1}_{T^Z_1<t}\EE_1(Z_{t-T^Z_1})] \\ 
    & \leq\EE_z(Z_{t\wedge T^Z_1})+1,
  \end{align*}
  where $T^Z_1$ is the first hitting time of 1 by $Z$. Hence, we only need to prove that $\sup_{z\geq 1}\EE_z(Z^{T_1}_t)<\infty$,
  where $Z^{T_1}$ is the diffusion process $Z$ stopped at time $T^Z_1$.

  Since $Z^{T_1}$ is a diffusion process on $[1,\infty)$, stopped at 1, on natural scale and with speed measure
  $\nu(\cdot)=\widetilde{m}(\cdot\cap(1,\infty))$ satisfying
  $$
  \nu([1,2])<\infty,
  $$
  we can apply the result of~\cite[Thm.\,4.1,\,Ex.\,2]{matsumoto-89}. This result ensures that, under
  conditions~\eqref{eq:matsumoto-1} and~\eqref{eq:matsumoto-bis}, the probability density function of $Z^{T_1}_t$ with respect to
  $\nu$, denoted by $p(x,y,t)$, is well defined for all $t>0$ and there exists a constant $A'_t>0$ such that
  \begin{align*}
    \sup_{1\leq x,y <+\infty} p(x,y,t)\leq A'_t,\quad \forall t>0.
  \end{align*}
  As a consequence, for all $t>0$,
  \begin{align*}
    \E_z(Z^{T_1}_t) & =\int_1^\infty y\, p(z,y,t)\,d\nu(y)\leq A'_t \int_1^\infty y\,d\widetilde{m}(y). \qedhere
  \end{align*}
\end{proof}

\subsection{Links between strict local martingales and absorbed diffusions}
\label{sec:links-QSD-strict-mg}

The next result gives a change of measure linking the behavior near 0 and near $+\infty$ of diffusion processes on $[0,\infty)$ on
natural scale. Note that it does not require that $X$ comes down from infinity nor that it hits 0 in finite time.

\begin{thm}
  \label{thm:h-transfo}
  Let $X$ be a diffusion process on $[0,+\infty)$ on natural scale stopped at 0 with speed measure $m$. We denote by $\P_x$ its law
  when $X_0=x$, defined on the canonical space of continuous functions from $\R_+$ to itself, equipped with its canonical filtration
  $(\mathcal{F}_t)_{t\geq 0}$. For all $x>0$, we define the following $h$-transform of $\P_x$:
  $$
  \restriction{d\widetilde{\mathbb{Q}}_x}{\mathcal{F}_\tau}=\frac{X_\tau}{x}  \restriction{d\mathbb{P}_x}{\mathcal{F}_\tau},
  $$
  for all $(\mathcal{F}_t)$-stopping time $\tau$ such that $(X_{t\wedge\tau},t\geq 0)$ is a uniformly integrable martingale. Then,
  under $(\widetilde{\mathbb{Q}}_x)_{x>0}$, the process $Z_t:=1/X_t$ is a diffusion process stopped at 0 on natural scale with speed
  measure
  \begin{equation}
    \label{eq:def-tilde-m}
    \widetilde{m}(dz)=\frac{1}{z^2}(f*m)(dz),
  \end{equation}
  where $f*m$ is the pushforward measure of $m$ by the application $f(x)=1/x$.
\end{thm}

In particular, if the measure $m(dx)$ is absolutely continuous w.r.t.\ Lebes\-gue's measure $\Lambda$ on $(0,+\infty)$, then
$\widetilde{m}(dz)$ is absolutely continuous w.r.t.\ $\Lambda$ and
\begin{equation}
  \label{eq:tilde-m-density}
  \frac{d\widetilde{m}}{d\Lambda}(z)=\frac{1}{z^4}\frac{dm}{d\Lambda}\left(\frac{1}{z}\right),\quad\forall z>0.  
\end{equation}

\begin{proof}
  The fact that $\widetilde{\QQ}_x$, $x>0$, are probability measures and that the process $Z=1/X$ is a diffusion on $[0,\infty)$
  stopped at 0 on natural scale under $(\widetilde{\QQ}_x)_{x>0}$ are proved in~\cite{petrowski-ruf-2012} (see Lemma~2.5 for the last
  point). Hence, we only have to check that $\widetilde{m}$ defined in~\eqref{eq:def-tilde-m} is its speed measure. By Green's
  formula~\cite[Lemma 23.10]{kallenberg-02}, this means that, for all $0<a<z<b<+\infty$ and for all bounded measurable function
  $g:(0,\infty)\rightarrow\R$,
  \begin{equation}
    \label{eq:goal}
    \E^{\widetilde{\QQ}_{1/z}}\left[\int_0^{T^Z_{a,b}}g(Z_t)\,dt\right]=2\int_a^b\frac{(z\wedge y-a)(b-z\vee y)}{b-a}g(y)\,\widetilde{m}(dy),    
  \end{equation}
  where $T^Z_{a,b}:=\inf\{t\geq 0: Z_t\in\{a,b\}\}$.
  
  On the one hand, the definition of $\widetilde{m}$ and the change of variable formula entails that the right hand term of~\eqref{eq:goal} is equal to
  \begin{multline*}
    2\int_a^b\frac{(z\wedge y-a)(b-z\vee y)}{b-a}\frac{g(y)}{y^2}\,(f*m)(dy) \\
    \begin{aligned}
      & =2\int_{1/b}^{1/a}\frac{(z\wedge \frac{1}{x}-a)(b-z\vee\frac{1}{x})}{b-a}x^2g(1/x)\,m(dx) \\ 
      & =2z\int_{1/b}^{1/a}\frac{(\frac{1}{a}-\frac{1}{z}\vee x)(x\wedge\frac{1}{z}-\frac{1}{b})}{\frac{1}{a}-\frac{1}{b}}xg(1/x)\,m(dx) \\ 
      & =z\E_{1/z}\left[\int_0^{T_{1/b,1/a}}X_tg(1/X_t)\,dt\right],
    \end{aligned}
  \end{multline*}
  where the last equality comes from the fact that $X$ is a diffusion on natural scale with speed measure $m$ under $(\P_x)_{x\geq
    0}$. 
    
    On the other hand, since $(X_{t\wedge T_{1/b,1/a}},t\geq 0)$ is a bounded martingale under $\PP_{1/z}$, we obtain by definition of $(\widetilde{\QQ}_x)_{x>0}$
  \begin{align*}
    \E^{\widetilde{\QQ}_{1/z}}\left[\int_0^{T^Z_{a,b}}g(Z_t)\,dt\right] & =\E^{\widetilde{\QQ}_{1/z}}\left[\int_0^{T_{1/b,1/a}}g(1/X_t)\,dt\right] \\
    & =z\E_{1/z}\left[X_{T_{1/b,1/a}}\int_0^{T_{1/b,1/a}}g(1/X_t)\,dt\right].
  \end{align*}
  Now, It\^o's formula implies that, a.s. for all $t\geq 0$,
  $$
  X_t\int_0^t g(1/X_s)\,ds=\int_0^t\left(\int_0^s g(1/X_u)\,du\right)\,dX_s+\int_0^t X_sg(1/X_s)\,ds.
  $$
  Since $(X_{t\wedge T_{1/b,1/a}},t\geq 0)$ is a bounded martingale under $\P_{1/z}$, for all $n\geq 1$,
  $$
  \E_{1/z}\left[X_{n\wedge T_{1/b,1/a}}\int_0^{n\wedge T_{1/b,1/a}}g(1/X_s)\,ds\right]
  =\E_{1/z}\left[\int_0^{n\wedge T_{1/b,1/a}}X_sg(1/X_s)\,ds\right].
  $$
  Since $\E_{1/z}[T_{1/b,1/a}]<\infty$ (this is a general property of diffusions), Lebesgue's theorem implies~\eqref{eq:goal}, which
  ends the proof of Theorem~\ref{thm:h-transfo}.
\end{proof}

The nature (regular/exit or entrance) of the boundary points 0 and $\infty$ for $X$ and $Z$ are related.

\begin{cor}
  \label{cor:h-transfo-1}
  With the previous notation, $X$ hits 0 in a.s.\ finite time iff $Z$ comes down form infinity, and $X$ comes down from infinity iff
  $Z$ hits 0 in a.s.\ finite time. In particular,
  \begin{align*}
    (Z_t)_{t\geq 0}\text{ is a strict local martingale}\ \Longleftrightarrow\ (X_t)_{t\geq 0} \mbox{ hits 0 in a.s.\ finite time}.
  \end{align*}
\end{cor}

\begin{proof}
  We need to check that
  $$
  \int_0^1 y\,m(dy)<\infty\quad\Longleftrightarrow\quad \int_1^\infty y\,\widetilde{m}(dy)<\infty
  $$
  and
  $$
  \int_1^\infty y\,m(dy)<\infty\quad\Longleftrightarrow\quad \int_0^1 y\,\widetilde{m}(dy)<\infty.
  $$
  This follows from the definition~\eqref{eq:def-tilde-m} of $\widetilde{m}$ by a simple change of variable.
\end{proof}

The next corollary gives an important relationship between the absorption probability of $X$ under $\P_x$ and the expectation of $Z$
under $\widetilde{\QQ}_{x}$. We obtain in particular a condition on the diffusion $X$ equivalent to the fact that $Z$ is a strict
local martingale in the strong sense.

\begin{cor}
\label{cor:h-transfo-2}
With the previous notation, for all $t>0$ and $x>0$, we have
\begin{align*}
\frac{\P^X_x(t<\tau_\d)}{x}=\E^Z_{1/x}(Z_t),
\end{align*}
where $\P^X_x$ and $\P^Z_z$ are the respective distributions of the diffusion processes $X$ such that $X_0=x$ and $Z$ such that
$Z_0=z$. In particular, for any constant $A>0$ and $t>0$,
\begin{equation}
\label{eq:yh}
\sup_{z>0} \E^Z_z(Z_t)\leq A\ \Longleftrightarrow\ \P^X_x(t<\tau_\d)\leq Ax,\ \forall x\geq 0.
\end{equation}
\end{cor}

\begin{proof}[Proof of Corollary~\ref{cor:h-transfo-2}]
We have, for all $x>0$ and all $t>0$,
\begin{equation*}
\E^Z_{1/x}(Z_t)=\E^{\widetilde{\QQ}_x}\left(\frac{1}{X_t}\right)=\frac{1}{x}\P^X_x(X_t>0).\qedhere
\end{equation*}
\end{proof}

\section{Quasi-stationary distribution for one dimensional diffusions}
\label{sec:1d_diffusions}
\label{sec:QSD-diff}

We consider in this section a diffusion process $X$ on $[0,\infty)$ on natural scale hitting 0 in a.s. finite time $\tau_\d$. Recall
that this is equivalent to assuming
$$
\int_0^1 y\,m(dy)<\infty
$$
where $m$ is the speed measure of $X$. We also set $X_t=0$ for all $t>\tau_\d$.

\subsection{Exponential convergence to quasi-stationary distribution}
\label{sec:QSD-without-killing}

The aim of this section is to establish that the following Condition~(B) is equivalent to the exponential convergence to a unique
quasi-stationary distribution of conditional distributions of $X$.

\begin{description}
\item[\textmd{(B)}] The process $X$ comes down from infinity and there exist two constants $t_1,A>0$ such that
  \begin{equation}
    \label{eq:hyp-diffusion}
    \P_x(\tau_\d> t_1)\leq Ax,\quad\forall x>0.
  \end{equation}
\end{description}

Let us recall that $X$ comes down from infinity iff $\int_1^\infty y\,m(dy)<\infty$.

\begin{rem}
  \label{rem:jolie-rk}
  Note that it is well known that, for diffusion processes on natural scale on $[0,\infty)$, $\P_x(\tau_\d<\infty)>0$ for some $x>0$
  implies that $\P_x(\tau_\d<\infty)=1$ for all $x>0$ (it can be also deduced from~\eqref{eq:str-loc-mg} and the results of
  Section~\ref{sec:links-QSD-strict-mg}). Hence~\eqref{eq:hyp-diffusion} implies that $\int_0^1 y\,m(dy)<\infty$.
\end{rem}

\begin{rem}
  \label{rem:cond-(B)}
  The converse inequality in~\eqref{eq:hyp-diffusion} is true for all diffusion process $X$ on natural scale on $(0,+\infty)$: for
  all $t>0$, there exists a constant $a>0$ such that
  $$
  \PP_x(t<\tau_\d)\geq a x,\quad\forall x\in(0,1).
  $$
  Indeed, by the strong Markov property at time $T_1$,
  \begin{align*}
    \PP_x(t<\tau_\d)\geq\EE_x[\mathbbm{1}_{T_1<\tau_\d}\PP_1(t<\tau_\d)]=\PP_x(T_1<\tau_\d)\PP_1(t<\tau_\d)=x \PP_1(t<\tau_\d).
  \end{align*}
\end{rem}

The following result is proved in Section~\ref{sec:proof-QSD-diff}.

\begin{thm}
  \label{thm:QSD_full}
  Assume that $X$ is a one-dimensional diffusion on natural scale on $[0,\infty)$ with speed measure $m(dx)$ such that
  $\tau_\d<\infty$ a.s. Then we have equivalence between
  \begin{description}
  \item[\textmd{(i)}] Assumption \textup{(B)}.
  \item[\textmd{(ii)}] There exist a probability measure $\alpha$ on $(0,\infty)$ and two constants $C,\gamma>0$ such that, for all
    initial distribution $\mu$ on $(0,\infty)$,
    \begin{align}
      \label{eq:expo-cv}
      \left\|\PP_\mu(X_t\in\cdot\mid t<\tau_\partial)-\alpha(\cdot)\right\|_{TV}\leq C e^{-\gamma t},\ \forall t\geq 0.
    \end{align}
  \end{description} 
  
  In this case, $\alpha$ is the unique quasi-stationary distribution for the process and there exists $\lambda_0>0$ such that
  $\P_\alpha(t<\tau_\d)=e^{-\lambda_0 t}$. Moreover, $\alpha$ is absolutely continuous with respect to $m$ and
  \begin{align}
    \label{eq:expression-of-alpha}
    \frac{d\alpha}{dm}(x)=2\lambda_0 \int_0^\infty (x\wedge y)\,\alpha(dy).
  \end{align}
  In addition,
  $$
  \int_0^\infty y\,\alpha(dy)<\infty.
  $$
\end{thm}

In the next result, we provide additional information on the function $\eta$ defined in~\eqref{eq:convergence-to-eta}.

\begin{prop}
  \label{prop:2}
  Assume that $X$ is a one-dimensional diffusion on natural scale with speed measure $m(dx)$ such that $\tau_\d<\infty$ a.s. and such
  that (B) is satisfied. Then the limit $\eta$ in~\eqref{eq:convergence-to-eta} is given by the normalized right hand side
  of~\eqref{eq:expression-of-alpha}:
  \begin{align}
    \label{eq:formula_eta}
    \eta(x)=\frac{\displaystyle{\int_0^\infty (x\wedge y)\,\alpha(dy)}}{\displaystyle{\int_0^\infty\int_0^\infty(y\wedge z)\,\alpha(dy)\,\alpha(dz)}}.
  \end{align}
  Moreover, the function $\eta$ is non-decreasing, bounded, $\alpha(\eta)=1$ and $\eta(x)\leq C x$ for all $x\geq 0$ and some constant
  $C>0$.
\end{prop}

\begin{rem}
  \label{rem:densite-alpha}
  We deduce from the last results that the density of $\alpha$ with respect to $m$ is proportional to $\eta$. Hence it is positive,
  bounded, (strictly) increasing on $(0,+\infty)$, and differentiable at each point of $(0,+\infty)$.
\end{rem}

\subsection{Conditions ensuring~(B) using elementary probabilistic tools}
\label{sec:condition-(B)}

The next result gives a condition implying (B). Its proof is particularly simple. We give in the next section a slightly more general
criterion which makes use of the results on strict local martingales of Section~\ref{sec:pties-diff}.

\begin{thm}
  \label{thm:SDE1}
  With the previous notation, assume that $X$ comes down from infinity, hits 0 a.s.\ in finite time and for all $x\in(0,1)$,
  \begin{equation}
    \label{eq:cond-moments}
    I(x):=\int_0^x y\,m(dy)\leq Cx^\rho    
  \end{equation}
  for some constants $C>0$ and $\rho>0$.
  Then, for all $t>0$, there exists $A_t<\infty$ such that
  \begin{align*}
    \P_x(t<\tau_\d)\leq A_t\, x,\ \forall x>0.
  \end{align*}
\end{thm}

We immediately deduce the next corollary.
\begin{cor}
  \label{cor:youpi}
  Assume that $X$ comes down from infinity and that $0$ is regular, i.e.\ $\int_0^1 m(dx)<\infty$. Then Condition~(B) is
  satisfied.
\end{cor}

\begin{rem}
  \label{rem:youpi}
  Equation~\eqref{eq:cond-moments} is satisfied if $m$ is absolutely continuous w.r.t.\ Lebesgue's measure $\Lambda$ on $(0,\infty)$
  and if
  $$
  \frac{dm}{d\Lambda}(x)\leq \frac{C}{x^{\alpha}}
  $$
  in the neighborhood of $0$, where $\alpha<2$. This corresponds to a diffusion process solution to the SDE
  $$
  dX_t=\sigma(X_t)dB_t
  $$
  with $\sigma(x)\geq cx^{\alpha/2}$ in the neighborhood of $0$.
\end{rem}

\begin{proof}
The proof is based on the following lemma.

\begin{lem}
  \label{prop:moment}
  For all $k\geq 1$ and $x\in\R_+$, the function $M_k(x)=
  \E_x\left(\tau_\d^k\right)$  is bounded and satisfies
  \begin{align}
    \label{eq:eq32}
    M_k(x)=2k\int_0^\infty (x\wedge y)\, M_{k-1}(y)\,m(dy)\, <\infty.
  \end{align}
\end{lem}

We admit this lemma for the moment and give its proof at the end of the subsection.

First, notice that we can assume without loss of generality that $\rho\not\in\QQ$. We prove by induction that there exists a
constant $C_k>0$ such that for all $x\in(0,1)$,
\begin{align}
  \label{eq:borne-moments}
  M_k(x)\leq C_k x^{(k\rho)\wedge 1}.
\end{align}
This is enough to conclude since one can choose $k$ such that $k\rho>1$ and use Markov's inequality to obtain
$$
\P_x(t<\tau_\d)\leq\frac{M_k(x)}{t^k}\leq C_t x
$$
where the constant $C_t$ depends on $t$ but not on $x$.

Equation~\eqref{eq:borne-moments} is trivial for $k=0$. Let us assume it is true for $k\geq 0$ and let us denote by $\alpha_k$ the
constant $(k\rho)\wedge 1$. For all $x\in(0,1)$, by~\eqref{eq:eq32},
\begin{align*}
M_{k+1}(x)&\leq C\left(\int_0^x y M_k(y)m(dy)+x\int_x^1 M_k(y)m(dy)+x\int_1^\infty M_k(y)m(dy)\right)\\
&\leq C\left(\int_0^x y^{1+\alpha_k}\,m(dy)+x\int_x^1 y^{\alpha_k-1}\,y\,m(dy)+x\int_1^\infty  m(dy)\right) \\
&\leq C \left[x^{\alpha_k} I(x)+ x\left(I(1)+(1-\alpha_k)\int_x^1 y^{\alpha_k-2}I(y)\,dy\right)+x\int_1^\infty m(dy)\right] \\
&\leq C x^{\alpha_k+\rho}+Cx +Cx\int_x^1 y^{\alpha_k+\rho-2}\,dy,
\end{align*}
where the constant $C>0$ depends on $k$ and may change from line to line, and the third inequality follows from integration by parts.
Since $\rho\not\in\QQ$, we have
$$
\int_x^1 y^{\alpha_k+\rho-2}\,dy\leq
\begin{cases}
  1/(\alpha_k+\rho-1) & \text{if }\alpha_k+\rho-1>0, \\
  x^{\alpha_k+\rho-1}/(1-\alpha_k-\rho) & \text{if }\alpha_k+\rho-1<0.
\end{cases}
$$
This concludes the induction.
\end{proof}

\begin{proof}[Proof of Lemma~\ref{prop:moment}]
Let us define the functions $M_k(x)$ recursively from the formula ~\eqref{eq:eq32} and $M_0(x)=1$. Our aim is hence to prove that, for all $k\geq 0$, $M_k$ is uniformly bounded and that, for all $x\in(0,+\infty)$, that $M_k(x)=\E_x(\tau_\d^k)$.

An immediate induction procedure
and~\eqref{eq:ymdy_integrable} proves that, for all $k\geq 0$, the function $M_k$ is bounded.

The Green function for diffusions on natural scale~\cite[Lemma 23.10]{kallenberg-02} implies that, for any $0<a<b<\infty$,
\begin{align*}
\E_x\left(\int_0^{T_{a,b}}M_{k-1}(X_t)\,dt\right)=2\int_a^b \frac{(x\wedge y-a)(b-x\vee y)}{b-a}M_{k-1}(y)\,m(dy).
\end{align*}
Now, letting $a\rightarrow 0$ (by monotone convergence) and $b\rightarrow \infty$ (by monotone convergence in the l.h.s and by Lebesgue's Theorem in the right hand side - remember that $M_{k-1}$ is bounded), we obtain
\begin{align*}
\E_x\left(\int_0^{\tau_\d}M_{k-1}(X_t)\,dt\right)=2\int_0^\infty (x\wedge y)\,M_{k-1}(y)\,m(dy)
\end{align*}
and hence
\begin{align*}
M_k(x)=k\,\E_x\left(\int_0^{\tau_\d}M_{k-1}(X_t)\,dt\right).
\end{align*}

Let us now prove by induction that $M_k(x)=\E(\tau_\d^k)$ for all $k\geq 0$. The property is true for $k=0$. Assume that $M_{k-1}(x)=\E_x(\tau_\d^{k-1})$, then
\begin{align*}
M_k(x)=k\E_x\left(\int_0^{\tau_\d}\E_{X_t}(\tau_\d^{k-1})\,dt\right)=k\E_x\left(\int_0^{\tau_\d} \left(\tau_\d-t\right)^{k-1}\,dt\right)=\E_x(\tau_\d^k),
\end{align*}
where we used the Markov property.
\end{proof}

\subsection{Conditions ensuring~(B) in relation with strict local martingales}
\label{sec:strict-local-martingales}

The next result gives a sufficient condition for (B), which is a weaker requirement than the one of Theorem~\ref{thm:SDE1}. It is a direct consequence of
Theorem~\ref{thm:h-transfo-3-v2}, the definition~\eqref{eq:def-tilde-m} of $\widetilde{m}$ and the equivalence~\eqref{eq:yh}.

\begin{thm}
  \label{thm:h-transfo-3}
  Assume that $X$ comes down from infinity and
  \begin{equation}
    \label{eq:matsumoto}
    \int_0^1\frac{1}{x}\,\sup_{y\leq x}\left(\frac{1}{y}\int_{(0,y)} z^2\,m(dz)\right)\,dx<\infty.
  \end{equation}
  Then, for all
  $t>0$, there exists $A_t<\infty$ such that
  \begin{align*}
    \P_x(t<\tau_\d)\leq A_t\, x,\ \forall x>0.
  \end{align*}
\end{thm}

Note that Fubini's theorem implies that
\begin{align*}
  \frac{1}{2}\int_0^{1/2}z\,m(dz)\leq \int_0^1 z(1-z)\,m(dz) & =\int_0^1 \frac{1}{x^2}\int_0^x z^2\,m(dz) \\
  & \leq\int_0^1\frac{1}{x}\left(\sup_{y\leq x}\frac{1}{y}\int_0^y z^2\,m(dz)\right)dx.
\end{align*}
Hence~\eqref{eq:matsumoto} implies that $\int_0^1 z\,m(dz)<\infty$, i.e.\ $X$ hits 0 in a.s.\ finite time.

The last theorem is based (via Theorem~\ref{thm:h-transfo-3-v2}) on abstract but powerful general analytical results on the density
of diffusion processes, which are hard to extend to more general settings. But Condition (B) is simple enough to be checked for other
classes of models using elementary probability tools, as we did in Theorem~\ref{thm:SDE1}. We give in Subsection~\ref{sec:jump} a
simple example of a one-dimensional process with jumps where exponential convergence to the quasi-stationary distribution can be
proved by an easy extension of Theorem~\ref{thm:SDE1}.

\begin{rem}
  \label{rem:matsumoto-bis}
  As in Remark~\ref{rem:matsumoto}, in the case of a (weak) solution to the SDE $dX_t=\sigma(X_t)dB_t$, the
  condition~\eqref{eq:matsumoto} covers almost all the practical situations where $\int_0^\infty y\sigma^{-2}(y)\,dy<\infty$. For
  example, it is true if $X$ comes down from infinity and
  \begin{align*}
    \sigma(x)\geq Cx\sqrt{\log \frac{1}{x}\cdots \log_{k-1} \frac{1}{x}\left(\log_k \frac{1}{x}\right)^{1+\epsilon}}
  \end{align*} 
  for all $x$ in a neighborhood of $0$ and for some $\epsilon,C>0$ and $k\geq 1$.
\end{rem}

\subsection{Consequences of the expression of $\alpha$ as a function of $m$}
\label{sec:dalpha-dm}

In Theorem~\ref{thm:QSD_full}, we proved that, under Assmuption~(B), $\alpha$ is absolutely continuous with respect to $m$ and that
\begin{align}
\label{eq:alpha-m}
\frac{d\alpha}{dm}(x)=2\lambda_0 \int_0^\infty (x\wedge y)\alpha(dy),\ \forall x\in(0,+\infty).
\end{align}
Our goal here is to provide a converse property, i.e.\ to give a necessary and sufficient condition on a given positive measure
$\alpha$, such that it is the unique quasi-stationary distribution of a diffusion process on natural scale on $[0,+\infty)$ with
exponential convergence of the conditional laws.

We will say (with a slight abuse of notation) that a positive measure $m$ on $(0,\infty)$ satisfies~(B) if: it is locally finite,
gives positive mass on each open subset of $(0,\infty)$ and the diffusion process on $[0,\infty)$ on natural scale with speed measure
$m$ and stopped at 0 satisfies Condition~(B).

\begin{thm}
  Fix $\lambda_0>0$ and a probability measure $\alpha$ on $(0,+\infty)$ such that the measure $(\frac{1}{x}\vee 1) \alpha(dx)$
  satisfies~(B). Then, there exists a unique (in law) diffusion process $(X_t)_{t\geq 0}$ on $[0,+\infty)$ on natural scale a.s.\
  absorbed at $0$ satisfying
  \begin{align*}
    \left\|\P_x(X_t\in\cdot\mid t<\tau_\d)-\alpha\right\|_{TV}\leq Ce^{-\gamma t},\quad \forall x\in (0,+\infty),
  \end{align*}
  for some positive constants $C,\gamma$ and such that
  \begin{align*}
    \P_\alpha(t<\tau_\d)=e^{-\lambda_0 t}, \quad \forall t\geq 0.
  \end{align*}
  In addition, the speed measure of the process $X$ is given by
  \begin{align*}
    m(dx)=\frac{\alpha(dx)}{2\lambda_0\int_0^\infty (x\wedge y) \alpha(dy)},\quad \forall x\in(0,\infty).
  \end{align*}
  Conversely, if a speed measure $m$ satisfies (B), then the corresponding quasi-stationary distribution $\alpha$ is such that
  $(\frac{1}{x}\vee 1) \alpha(dx)$ satisfies~(B).
\end{thm}

\medskip
Because of the relation~\eqref{eq:alpha-m} and the equivalence between (B) and uniform exponential convergence to a quasi-stationary distribution, this result follows from the next lemma.

\begin{lem}
\label{prop:m-alpha}
Let $\alpha$ be a probability measure on $(0,+\infty)$. Define the measure $m$ on $(0,+\infty)$ as
\begin{align}
\label{eq:m-alpha}
m(dx)=\frac{\alpha(dx)}{\int_0^\infty (x\wedge y) \alpha(dy)},\ \forall x\in(0,\infty).
\end{align}
Then the measure $(\frac{1}{x}\vee 1) \alpha(dx)$ satisfies (B) if and only if the measure $m$ satisfies (B).
\end{lem}

\begin{proof}
  Note that, if $m_1,m_2$ are two locally finite measures on $(0,+\infty)$ giving positive mass to each open subset of $(0,\infty)$,
  such that $m_1\leq C m_2$ for some positive constant $C>0$  and such that $m_2$ satisfies (B), then $m_1$ also satisfies (B). This is a direct consequence of the
  construction of diffusion processes given in Section~\ref{sec:construction} and of~\eqref{eq:As}.

  Assume first that $(\frac{1}{x}\vee 1) \alpha(dx)$ satisfies (B). Then, using Remark~\ref{rem:jolie-rk},
  \begin{align}
    \label{eq:alpha-integrable}
    \int_0^\infty (1\vee y)\,\alpha(dy)<\infty.
  \end{align}
  Thus
  \begin{align}
    \label{eq:preuve-lemme}
    \int_0^\infty (x\wedge y) \alpha(dy)\leq \left(x\int_0^\infty \alpha(dy)\right)\wedge\left(\int_0^\infty y\,\alpha(dy)\right)<\infty.
  \end{align}
  Hence $m\neq 0$, is locally finite and gives positive mass to each open subset of $(0,\infty)$.
  Since
  \begin{align*}
    \int_0^\infty (x\wedge y)\alpha(dy)\geq
    \begin{cases}
      x\int_1^\infty \alpha(dy),&\text{ if }x\leq 1,\\
      \int_1^\infty\alpha(dy),&\text{ if }x>1,
    \end{cases}
  \end{align*}
  we have
  $$
  \frac{dm}{d\alpha}(x)=\left(\int_0^\infty (x\wedge y) \alpha(dy)\right)^{-1}\leq\left(\int_1^\infty \alpha(dy)\right)^{-1}
  \left(\frac{1}{x}\vee 1\right). 
  $$
  Therefore, if $(\frac{1}{x}\vee 1) \alpha(dx)$ satisfies (B), $m$ satisfies (B).

  Conversely, if $m$ satisfies (B), Theorem~\ref{thm:QSD_full} implies that $\alpha$ is proportional to the unique quasi-stationary
  distribution of $X$ with speed measure $m$ and satisfies
  $$
  \int_0^\infty y\,\alpha(dy)<\infty.
  $$
  Since $\alpha$ is a finite measure, we deduce from~\eqref{eq:preuve-lemme} that
  $$
  \int_0^\infty (x\wedge y) \alpha(dy)\leq C(x\wedge 1)
  $$
  for some constant $C<\infty$. Hence
  $$
  \frac{dm}{d\alpha}(x)\geq C^{-1}\left(\frac{1}{x}\vee 1\right),
  $$
  i.e.\ $(\frac{1}{x}\vee 1) \alpha(dx)$ satisfies (B).
\end{proof}

\subsection{Examples}
\label{sec:examples}

\subsubsection{On general diffusions}
\label{sec:ex1}

Let us first recall that our results also cover the case of general (i.e.\ not necessarily on natural scale) diffusion processes
$(Y_t,t\geq 0)$ on $[a,b)$ such that $-\infty <a<b\leq+\infty$ with speed measure $m_Y$, which hits $a$ in a.s.\ finite time $\tau_\d$,
absorbed at $a$ and such that $b$ is a natural or entrance boundary. Under these assumptions, there exists a continuous and strictly
increasing scale function $s:[a,b)\rightarrow[0,\infty)$ of the process $Y$ such that $s(a)=0$ and $s(b-)=+\infty$. Our results then
apply to the process $X_t=s(Y_t)$ which is a diffusion on natural scale on $[0,\infty)$, whose speed measure is given by
\begin{align}
  \label{eq:m_X}
  m_X(dx)=s*m_Y(dx),
\end{align}
where $s*m_Y$ is the pushforward measure of $m_Y$ through the function $s$ (this follows for instance
from~\cite[Thm.\,VII.3.6]{Revuz1999}).

Hence $X$ satisfies (B) if and only if $Y$ satisfies the following Condition~(B').

\bigskip\noindent\textbf{Condition (B')} Assume that $\int_c^b s(y)\,m_Y(dy)<\infty$ for some $c\in(a,b)$ and there exist two
constants $t_1,A>0$ such that
\begin{equation*}
  \PP(t_1<\tau_\d\mid Y_0=y)\leq As(y),\quad\forall y>a.
\end{equation*}
\smallskip

Under this assumption, the process $Y$ satisfies~\eqref{eq:expo-cv} for a unique quasi-stationary distribution $\alpha_Y$ on $(a,b)$ and $\alpha_Y$ satisfies
$$
\int_a^b s(y)\alpha_Y(dy)<\infty.
$$
Moreover
$$
\frac{d\alpha_Y}{dm_Y}(y)=2\lambda_0\int_a^b (s(y)\wedge s(z))\alpha_Y(dz),
$$
where $\lambda_0$ is such that $\PP(t<\tau_\d\mid Y_0\sim\alpha_Y)=e^{-\lambda_0 t}$. Theorem~\ref{thm:h-transfo-3} entails the
following result.

\begin{prop}
  \label{prop:Y-ne-e}
  With the previous notation, if $\int_a^b s(y)\,m_Y(dy)<\infty$ and
  \begin{equation}
    \label{eq:Y-ne-e}
    \int_0^1\frac{1}{x}\left(\sup_{y\in(a,s^{-1}(x)]}\frac{1}{s(y)}\int_a^y s(z)^2\,m_Y(dz)\right)dx<\infty,
  \end{equation}
  then Condition~(B') is satisfied.
\end{prop}

Let us also mention that our methods can be easily extended to diffusion processes on a bounded interval $[a,b]$ where both boundary
points are either exit or regular (reflecting or absorbing). In the case where $b$ is exit or regular absorbing (i.e.\ $\tau_\d$ is
the first hitting time of $\{a,b\}$), we obtain two conditions of the form~\eqref{eq:Y-ne-e} in the neighborhood of $a$ and $b$.

To illustrate the generality of the processes that are covered by our criteria, we give two simple examples where the speed measure
$m$ is singular with respect to Lebesgue's measure $\Lambda$. The case of speed measures absolutely continuous w.r.t.\ Lebesgue's
measure (i.e.\ of SDEs) will be discussed in the next subsection. 

\paragraph{Example 1}

We recall that a diffusion process on $\RR$ with speed measure $\Lambda+\delta_0$ is called a \emph{sticky Brownian
  motion}~\cite{ito-mckean-74,amir-91}. We consider a diffusion process on $[0,\infty)$ which comes down from infinity ($\int_1^\infty
y\,m(dy)<\infty$) and which is a sticky Brownian motion on $[0,1]$, ``sticked'' at the points $a_1,a_2,\ldots$, where $(a_i)_{i\geq 1}$
is decreasing, converges to 0 and $a_1<1$, i.e.
$$
\restriction{m}{(0,1)}=\restriction{\Lambda}{(0,1)}+\sum_{i\geq 1}\delta_{a_i}.
$$
Assuming that there exist constants $C,\rho>0$ such that for all $j\geq 1$,
\begin{align}
\label{eq:lesai}
\sum_{i\geq j}a_i\leq C a^\rho_{j},
\end{align}
then for all $x\in (0,1)$, defining $i_0:=\inf\{j\geq 1: a_j< x\}$,
$$
\int_{(0,x)} y\,m(dy)=\frac{x^2}{2}+\sum_{i\geq i_0} a_i\leq \frac{x^2}{2}+Ca^\rho_{i_0}\leq \frac{x^2}{2}+Cx^\rho,
$$
and we can apply Theorem~\ref{thm:SDE1}.

For example, the choice $a_i=i^{-\frac{1}{1-\rho}}$, for all $i\geq 1$, satisfies~\eqref{eq:lesai}.

%
%

\paragraph{Example 2} Let $X$ be a sticky Brownian motion stopped at $-1$ and $1$. This means that $X$ is a diffusion on natural
scale with speed measure $m(dx)=\Lambda(dx)+\delta_0(dx)$ on $(-1,1)$, absorbed at $-1$ and $1$. As mentionned above, it is
straightforward to adapt our results to processes on $[-1,1]$. For the process of this example, the corresponding Assumption~(B) is
clearly fulfilled since both boundaries $-1$ and $1$ are regular for $X$ (see Corollary~\ref{cor:youpi}). Then the unique
quasi-stationary distribution $\alpha$ of $X$ is
\begin{align*}
\alpha(dx)=\frac{\gamma^*}{2}\sin\left(\gamma^*(1+x)\wedge(1-x)\right)m(dx),
\end{align*}
where $\gamma^*$ is the unique solution in $(0,\pi]$ of $\text{cotan }\gamma=\gamma/2$.

Indeed, it satisfies the following adaptation of formula~\eqref{eq:expression-of-alpha}
%
\begin{align}
\label{eq:abc}
\frac{d\alpha}{dm}(x)=\lambda_0\int_{-1}^1 (x\wedge y+1)(1-x\vee y)\,\alpha(dy),\ \forall x\in(-1,1).
\end{align}
In particular, the measure $\alpha$ is absolutely continuous with respect to Lebesgue's measure $\Lambda$ on $(-1,1)\setminus\{0\}$ and satisfies
\begin{align*}
\left(\frac{d\alpha}{d\Lambda}\right)''(x)=-2\lambda_0\frac{d\alpha}{d\Lambda}(x),\ \forall x\neq 0.
\end{align*}
Using the symmetry of the problem and the $0$ boundary conditions, we deduce that there exist constants $a,b\in\R$ such that
\begin{align*}
\alpha(dx)=a\delta_0(dx)+b\sin\left(\gamma(1+x)\wedge(1-x)\right)\Lambda(dx),
\end{align*}
where $\gamma=\sqrt{2\lambda_0}$. Note that the fact that $\alpha$ is a positive measure implies that $\gamma\in(0,\pi]$. In
addition, equality~\eqref{eq:abc} at $x=0$ entails
\begin{align*}
a=\lambda_0\left(2b\int_0^1 (1-y)\sin(\gamma(1-y))dy+a\right)
\end{align*}
and hence
\begin{align}
\label{eq:abc1}
(1-\gamma^2/2)a= b\left(\sin\gamma-\gamma\cos\gamma\right).
\end{align}
In addition, $d\alpha/dm(x)$ is continuous at $0$ by Remark~\ref{rem:densite-alpha}, hence
\begin{align}
\label{eq:abc2}
a=b\sin \gamma .
\end{align}
Finally, $\alpha$ is a probability measure, so that
\begin{align}
\label{eq:abc3}
a+\frac{2b}{\gamma}(1-\cos\gamma)=1.
\end{align}
Now, dividing~\eqref{eq:abc1} by~\eqref{eq:abc2}, we obtain $\gamma=\gamma^*$. Equality~\eqref{eq:abc3} becomes
\begin{align*}
1=a-b\sin\gamma+\frac{2b}{\gamma}.
\end{align*}
By~\eqref{eq:abc2}, we deduce that $b=\gamma/2$ and finally that $a=\frac{\gamma\sin\gamma}{2}$.

\subsubsection{On processes solutions of stochastic differential equations}

Let $Y$ be the solution of the general SDE on $[a,b)$ with $a<b\leq+\infty$
\begin{equation}
  \label{eq:EDS-generale}
  dY_t=\sigma(Y_t)dB_t+\mu(Y_t)dt,
\end{equation}
absorbed at $a$, where $\mu/\sigma^2\in L^1_{\text{loc}}((a,b))$ (which ensures the existence of a weak
solution~\cite[Ch.\,23]{kallenberg-02}, possibly exposive if $b=+\infty$, or which can possibly hit $b$ if $b<\infty$). In this case,
a scale function of $Y$ is given by
\begin{equation}
  \label{eq:scale-function-general}
  s(x)=\int_c^x\exp\left(-\int_c^y 2\mu(z)\sigma^{-2}(z)dz\right)dy,\ \forall x\in (a,b)  
\end{equation}
for any $c\in(a,b)$, and the speed measure by $m_Y(dy)=2/(s'(y)\sigma^2(y))\Lambda(dy)$. In the case where $s(a+)>-\infty$, one can
take $c=a$ in Formula~\eqref{eq:scale-function-general}, so that $s(a)=0$. If in addition $s(b-)=\infty$ (which implies in particular
that~\eqref{eq:EDS-generale} is non-explosive if $b=\infty$ or cannot hit $b$ in finite time if $b<\infty$), the
formula~\eqref{eq:scale-function-general} can be used to check the condition~\eqref{eq:Y-ne-e} in Proposition~\ref{prop:Y-ne-e}.

In particular, our criteria apply to cases of speed measures with discontinuous densities with respect to Lebesgue's measure, i.e.\
to cases of discontinuous $\sigma$. Such diffusions arise as rescaled solutions of SDEs driven by their local times, including skew
Brownian motions (see for instance~\cite{legall-83}). Hence our criteria also extend to one-dimensional SDEs driven by their local
time in the sense of~\cite{legall-83}.

Note that a major part of the existing literature~\cite{Cattiaux2009,littin-12,kolb-steinsaltz-12} on the subject only considers solutions to SDEs with absolutely 
continuous coefficient $\sigma$. The reason behind this is that all the previously cited works are based on Sturm-Liouville theory, which is particularly well suited for the study of diffusion processes $(Y_t,t\geq 0)$ on $[0,\infty)$  solution to a SDE of the form
\begin{equation*}
  \label{eq:EDS}
  dY_t=dB_t-q(Y_t)dt.
\end{equation*}
After a change of scale, this class of processes can only provide solutions to SDEs of the form~\eqref{eq:EDS-generale} with absolutely
continuous $\sigma$.

\subsubsection{A simple example with jumps}
\label{sec:jump}

The following simple example shows how our method can be easily extended to processes with jumps.

We consider a diffusion process $(X_t,t\geq 0)$ on $[0,\infty)$ with speed measure $m$ satisfying Assumption~(B).
Let us denote by $\mathcal{L}$ the infinitesimal generator of $X$. We consider the Markov process $(\tX_t,t\geq 0)$ with
infinitesimal generator
$$
\widetilde{\mathcal{L}}f(x)=\mathcal{L}f(x)+(f(x+1)-f(x))\mathbbm{1}_{x\geq 1},
$$
for all $f$ in the domain of $\mathcal{L}$. In other words, we consider a c\`adl\`ag process following a diffusion process with speed
measure $m$ between jump times, which occur at the jump times of an independent Poisson process $(N_t,t\geq 0)$ of rate $1$, with
jump size $+1$ if the process is above 1, and $0$ otherwise. We denote by $\widetilde{\tau}_\d$ its first hitting time of $0$.

The proof of the following proposition is postponed to Subsection~\ref{sec:proof-prop-jump}.

\begin{prop}
  \label{prop:jump}
  Under the previous assumptions, there exist a unique probability measure $\alpha$ on $(0,\infty)$ and two constants $C,\gamma>0$
  such that, for all initial distribution $\mu$ on $(0,\infty)$,
  \begin{align}
    \label{eq:expo-cv-jump}
    \left\|\PP_\mu(\widetilde{X}_t\in\cdot\mid t<\widetilde{\tau}_\partial)-\alpha(\cdot)\right\|_{TV}\leq C e^{-\gamma t},\ \forall t\geq 0.
  \end{align}
  In particular, the probability measure $\alpha$ is the unique quasi-stationary distribution of the process.
\end{prop}

\section{Proof of the results of Section~\ref{sec:1d_diffusions}}
\label{sec:proof-QSD-diff}

The main part of Theorem~\ref{thm:QSD_full}, the uniform convergence in~\eqref{eq:convergence-to-eta} and some of the results of
Proposition~\ref{prop:2} are proved using the criterion developed in~\cite{ChampagnatVillemonais2016} for general Markov
processes. More precisely, the following condition (A) is equivalent to Condition~(ii) of Theorem~\ref{thm:QSD_full}
(\cite[Thm.\,2.1]{ChampagnatVillemonais2016}), implies property~\eqref{eq:convergence-to-eta} and the fact that $\eta$ is bounded and
$\alpha(\eta)=1$. (\cite[Prop.\,2.3]{ChampagnatVillemonais2016}).

\paragraph{Assumption~(A)}
There exists a probability measure $\nu$ on $(0,\infty)$ such that
\begin{itemize}
\item[(A1)] there exist $t_0,c_1>0$ such that for all $x>0$,
  $$
  \PP_x(X_{t_0}\in\cdot\mid t_0<\tau_\partial)\geq c_1\nu(\cdot);
  $$
\item[(A2)] there exists $c_2>0$ such that for all $x>0$ and $t\geq 0$,
  $$
  \PP_\nu(t<\tau_\partial)\geq c_2\PP_x(t<\tau_\partial).
  $$
\end{itemize}

Hence, we need first to prove that (B) is equivalent to (A) (in Subsection~\ref{sec:proof-first-part}), second, to
prove~\eqref{eq:formula_eta} and that $\eta(x)\leq Cx$ (in Subsection~\ref{sec:proof-proposition}), and third to
prove~\eqref{eq:expression-of-alpha} and that $\int y\,\alpha(dy)<\infty$ (in Subsection~\ref{sec:proof-second-part}). We next give
in Subsection~\ref{sec:proof-prop-jump} the proof of Proposition~\ref{prop:jump} on processes with jumps.

\subsection{Proof of the equivalence between (A) and (B)}
\label{sec:proof-first-part}

Note that~\cite[Thm.\,2.1]{ChampagnatVillemonais2016} also assumes that
\begin{align}
  \label{eq:step0}
  \P_x(t<\tau_\d)>0,\quad \forall x>0,\quad \forall t>0,
\end{align}
so we need first to check that this is true for our diffusion processes. This follows easily from the strong Markov and the
regularity properties of $X$ (it is for example enough to use that $\P_x(T_{x/2}<+\infty)>0$ and $\P_{x/2}(T_x<+\infty)>0$).

We actually prove below that (A1) and (B) are equivalent and that these properties imply (A2). In other words, for one-dimensional
diffusion processes,~\eqref{eq:expo-cv} is actually equivalent to~(A1) alone.

We first prove that (A1) implies (B). If~(A1) holds true, there exist $0<a<b$ such that
  \begin{equation}
    \label{eq:A1}
    \inf_{x>0}\P_x(X_{t_1}\in [a,b]\mid t_1<\tau_\d)=\underline{c}>0.
  \end{equation}
  Now, for all $x>b$,
  $$
  \frac{\P_x(T_b<t_1)}{\P_b(t_1<\tau_\d)}\geq\frac{\P_x(X_{t_1}\in[a,b])}{\P_b(t_1<\tau_\d)}\geq\P_x(X_{t_1}\in[a,b]\mid t_1<\tau_\d).
  $$
  Since we proved above that $\P_b(t_1<\tau_\d)>0$, we deduce that $\inf_{x>b}\P_x(T_b<t_1)>0$, i.e.\ that $\infty$ is an entrance
  boundary for $X$. Equation~\eqref{eq:A1} also implies that, for all $x<a$,
  \begin{align*}
    \P_x(t_1<\tau_\d) & \leq\frac{\P_x(X_{t_1}\in [a,b])}{\underline{c}} \\
    & \leq\frac{\E_x(X_{t_1\wedge T_a})}{a\underline{c}}=\frac{x}{a\underline{c}}.
  \end{align*}
  Hence (i) is proved.

  The difficult part of the proof is the implication (B)$\Rightarrow$(A).

  \medskip
  \noindent
  \textit{Step 1: the conditioned process escapes a neighborhood of 0 in finite time.}\\
  The goal of this step is to prove that there exists $\varepsilon,c>0$ such that
  \begin{equation}
    \label{eq:step1}
    \P_x(X_{t_1}\geq\varepsilon\mid t_1<\tau_\d)\geq c,\quad\forall x>0.    
  \end{equation}

  To prove this, we first observe that, since $X$ is a local martingale, for all $x\in(0,1)$,
  $$
  x=\E_x(X_{t_1\wedge T_1})=\P_x(t_1<\tau_\d)\E_x(X_{t_1\wedge T_1}\mid t_1<\tau_\d)+\P_x(T_1<\tau_\d\leq t_1).
  $$
  By the Markov property,
  \begin{align*}
  \P_x(T_1<\tau_\d\leq t_1)&\leq \EE_x\left[\mathbbm{1}_{T_1<\tau_\d\wedge t_1}\P_1(\tau_\d\leq t_1)\right]\\
  &\leq \P_x(T_1<\tau_\d)  \P_1(\tau_\d\leq t_1)\\
  &=x \P_1(\tau_\d\leq t_1).
  \end{align*}
  Hence~\eqref{eq:hyp-diffusion} entails
  $$
  \E_x(1-X_{t_1\wedge T_1}\mid t_1<\tau_\d)\leq 1-\frac{1}{A'},
  $$
  where $A':=A/\P_1(t_1<\tau_\d)$. Note that, necessarily, $A'>1$. Markov's inequality then implies that, for all $x\in(0,1)$,
  \begin{equation}
    \label{eq:calcul2}
    \P_x(X_{t_1\wedge T_1}\leq \frac{1}{2A'-1}\mid t_1<\tau_\d)\leq \frac{1-1/A'}{1-1/(2A'-1)}=1-\frac{1}{2A'}.    
  \end{equation}
  Since $\P_{1/(2A'-1)}(t_1<\tau_\d)>0$, there exists $\varepsilon\in(0,1/(2A'-1))$ such that
  \begin{equation}
    \label{eq:def-epsilon}
    \P_{1/(2A'-1)}(t_1<T_\varepsilon)>0.    
  \end{equation}
  Hence, for all $x\in(0,1)$,
  \begin{align*}
    \P_x(X_{t_1}\geq\varepsilon) & \geq\P_x(T_{1/(2A'-1)}<t_1)\P_{1/(2A'-1)}(T_\varepsilon>t_1) \\ & \geq\P_x(X_{t_1\wedge T_1}\geq 1/(2A'-1))
    \P_{1/(2A'-1)}(T_\varepsilon>t_1) \\ & \geq \frac{\P_x(t_1<\tau_\d)}{2A'}\P_{1/(2A'-1)}(T_\varepsilon>t_1)
  \end{align*}
  by~\eqref{eq:calcul2}. This ends the proof of~\eqref{eq:step1} for $x<1$. For $x\geq 1>1/(2A'-1)>\varepsilon$, the continuity and
  the strong Markov property for $X$ entail
  $$
  \P_x(X_{t_1}>\varepsilon\mid t_1<\tau_\d)\geq\P_x(X_{t_1}>\varepsilon)\geq\P_x(T_\varepsilon>t_1)\geq\P_{1/(2A'-1)}(T_\varepsilon>t_1)>0.
  $$
  Hence~\eqref{eq:step1} is proved.

  \medskip
  \noindent
  \textit{Step 2: Construction of coupling measures for the unconditioned process.}\\
  Our goal is to prove that there exist two constants $t_2,c_1>0$ such that, for all $x\geq\varepsilon$,
  \begin{equation}
    \label{eq:step2}
    \P_x(X_{t_2}\in \cdot)\geq c_1\nu,
  \end{equation}
  where 
  $$
  \nu=\P_\varepsilon(X_{t_2}\in\cdot\mid t_2<\tau_\d).
  $$
  This kind of relations can be obtained with classical coupling arguments, which we detail here for completeness. Fix
  $x\geq\varepsilon$ and construct two independent diffusions $X^\varepsilon$ and $X^x$ with speed measure $m(dx)$, and initial value
  $\varepsilon$ and $x$ respectively. Let $\theta=\inf\{t\geq 0:X^\varepsilon_t=X^x_t\}$. By the strong Markov property, the process
  $$
  Y^x_t=
  \begin{cases}
    X^x_t & \text{if }t\leq\theta, \\
    X^\varepsilon_t & \text{if }t>\theta
  \end{cases}
  $$
  has the same law as $X^x$. Since $\theta\leq\tau^x_\d:=\inf\{t\geq 0: X^x_t=0\}$, for all $t>0$,
  $\P(\theta<t)\geq\P(\tau^x_\d<t)$. Since $\infty$ is an entrance boundary and $0$ is accessible for $X$, there exists $t_2>0$ such
  that
  $$
  \inf_{y>0}\P_y(\tau_\d<t_2)=c'_1>0.
  $$
  Hence
  $$
  \P_x(X_{t_2}\in\cdot)=\P(Y^x_{t_2}\in\cdot)\geq\P(X^\varepsilon_{t_2}\in\cdot,\ \tau^x_\d<t_2)\geq c'_1\P_\varepsilon(X_{t_2}\in \cdot).
  $$
  Therefore, \eqref{eq:step2} is proved with $c_1=c'_1\P_\varepsilon(t_2<\tau_\d)$.

  \medskip
  \noindent
  \textit{Step 3: Proof of (A1).}\\
  Using successively the Markov property, Step 2 and Step 1, we have for all $x>0$
  \begin{align*}
    \P_x(X_{t_1+t_2}\in\cdot\mid t_1+t_2<\tau_\d) & \geq \P_x(X_{t_1+t_2}\in\cdot\mid t_1<\tau_\d) \\ 
    & \geq \int_\varepsilon^\infty \P_y(X_{t_2}\in\cdot)\P_x(X_{t_1}\in dy\mid t_1<\tau_\d) \\
    & \geq c_1\int_\varepsilon^\infty \nu(\cdot)\P_x(X_{t_1}\in dy\mid t_1<\tau_\d) \\
    & =c_1\nu(\cdot)\P_x(X_{t_1}\geq\varepsilon\mid t_1<\tau_\d)\geq c_1 c \nu(\cdot).
  \end{align*}
  This entails (A1) with $t_0=t_1+t_2$. Hence we have proved the equivalence between (A1) and (B).

  \medskip
  \noindent
  \textit{Step 4: Proof of (A2).}\\  
  The general idea of the proof is close to the case of birth and death processes in~\cite{ChampagnatVillemonais2016}.  
  
  For all $0<a<b<\infty$, we have
  $$
  \E_x (T_a\wedge T_b)=2\int_a^b\frac{(x\wedge y-a)(b-x\vee y)}{b-a}\,m(dy),\quad \forall a<x<b.
  $$
  Hence,
  $$
  \E_x (T_a\wedge T_b)\leq 2\int_a^b(x\wedge y-a)\,m(dy)\leq 2 \int_a^\infty  y \,m(dy).
  $$
 Since the process is non-explosive, the left hand side converges to $\E_x(T_a)$ when $b\rightarrow\infty$. But $\int_0^\infty y\,m(dy)<\infty$ by assumption, so that, for all $\varepsilon>0$, there exists $a_\varepsilon>0$ such that
  \begin{align*}
    \sup_{x\geq a_\varepsilon}\E_x(T_{a_\varepsilon})\leq \varepsilon.
  \end{align*}
  Therefore, $\sup_{x\geq a_\varepsilon}\P_x(T_{a_\varepsilon}\geq 1)\leq\varepsilon$ and, applying recursively the Markov property,
  $\sup_{x\geq a_\varepsilon}\P_x(T_{a_\varepsilon}\geq k)\leq\varepsilon^k$. Then, for all $\lambda>0$, there exists $y_\lambda\geq 1$ such
  that
  \begin{equation}
    \label{eq:moment-expo}
    \sup_{x\geq y_\lambda}\E_x(e^{\lambda T_{y_\lambda}})<+\infty.  
  \end{equation}

The proof of the following lemma is postponed to the end of this section.
\begin{lem}
\label{le:lemme1}
There exists $a>0$ such that $\nu([a,+\infty[)>0$ and, for all $k\in\N$,
\begin{align*}
\P_a(X_{kt_0}\geq a)\geq e^{-\rho k t_0},
\end{align*}
with $\rho>0$.
\end{lem}

Take $a$ as in the previous lemma.
From~\eqref{eq:moment-expo}, one can choose $b> a$ large enough so that
\begin{align}
  \label{eq:moment-expo-2}
  \sup_{x\geq b}\E_x\left(e^{\rho T_{b}}\right)<\infty.
\end{align}
We are also going to make use of two inequalities. Since
\begin{align*}
\P_a(t<\tau_\d)\geq \P_a(T_b<s_0)\P_b(t<\tau_\d)
\end{align*}
for all $s_0\geq 0$, we obtain the first inequality: $\forall t\geq 0$,
\begin{align}
\label{eq:ineg1}
\sup_{x\in[a,b]} \P_x(t<\tau_\d)=\P_b(t<\tau_\d)\leq C\P_a(t<\tau_\d)=C\inf_{x\in[a,b]}\P_x(t<\tau_\d),
\end{align}
where $C$ is a positive constant. We recall that the regularity assumption ensures that $\P_a(T_b<s_0)>0$ for $s_0$ large enough. The
second inequality is an immediate consequence of the Markov property: $\forall s<t$,
\begin{align}
\P_a(X_{\lceil s/t_0 \rceil t_0}\geq a)\P_a(t-s<\tau_\d) & =
\P_a(X_{\lceil s/t_0 \rceil t_0}\geq a)\inf_{x\in[a,\infty)}\P_x(t-s<\tau_\d) \notag \\ & \leq \P_a(t<\tau_\d).
\label{eq:ineg2}
\end{align}

In the following computation, we use successively~\eqref{eq:moment-expo-2},~\eqref{eq:ineg1} and~\eqref{eq:ineg2}. For all $x\geq b$,
with a constant $C>0$ that may change from line to line,
\begin{align*}
\P_x(t<\tau_\d)&\leq \P_x(t< T_b)+\int_{0}^t \P_b(t-s<\tau_\d)\,\P_x(T_b\in ds)\\
&\leq Ce^{-\rho t}+C\int_0^t \P_a(t-s<\tau_\d)\,\P_x(T_b\in ds)\\
&\leq Ce^{-\rho (\lceil t/t_0 \rceil -1)t_0}+C \P_a(t<\tau_\d)\int_0^t \frac{1}{\P_a(X_{\lceil s/t_0 \rceil t_0}\geq a)}\,\P_x(T_b\in ds)\\
&\leq C\P_a(t<\tau_\d)+C \P_a(t<\tau_\d)\int_0^t e^{\rho(s+t_0)}\,\P_x(T_b\in ds),
\end{align*}
where we used twice Lemma~\ref{le:lemme1} in the last line. We deduce form~\eqref{eq:moment-expo-2} that, for all $t\geq 0$,
\begin{align*}
\sup_{x\geq b}\P_x(t<\tau_\d)\leq C \P_a(t<\tau_\d).
\end{align*}
Since $\nu([a,+\infty[)>0$, this ends the proof of~(A2).

\begin{proof}[Proof of Lemma~\ref{le:lemme1}]
Fix $a>0$ such that $\nu([a,+\infty[)>0$. Step 3 of the previous proof and~\eqref{eq:step0} entail
\begin{align*}
\P_a(X_{t_0}\geq a)\geq c_1\nu([a,+\infty[)\P_a(t_0<\tau_\d)\stackrel{\text{def}}{=}e^{-\rho t_0}>0.
\end{align*}
Now, using Markov property,
\begin{align*}
\P_a(X_{kt_0}\geq a)\geq \left(\inf_{x\geq a} \P_x(X_{t_0}\geq a)\right)^k.
\end{align*}
Since $\inf_{x\geq a} \P_x(X_{t_0}\geq a)=\P_a(X_{t_0}\geq a)$ by coupling arguments, the proof is completed.
\end{proof}

\subsection{Proof of Proposition~\ref{prop:2}}
\label{sec:proof-proposition}

Since (A) is equivalent to (B), the convergence in~\eqref{eq:convergence-to-eta} for the uniform normis a direct consequence of
Proposition~\cite[Prop~2.3]{ChampagnatVillemonais2016}. This proposition also entails that $\eta$ is bounded, positive on
$(0,+\infty)$ and that $\alpha(f)=1$. The fact that $\eta$ is non-decreasing comes from~\eqref{eq:convergence-to-eta} and from the
fact that $\PP_x(t<\tau_\d)$ is non-decreasing in $x\geq 0$ by standard comparison arguments. The fact that $\eta(x)\leq Cx$ follows
from assumption (B) since
$$
P_{t_1}\eta(x)\leq \|\eta\|_\infty\PP_x(t_1\leq\tau_\d)\leq \|\eta\|_\infty A x,\quad\forall x\geq 0.
$$

It only remains to prove~\eqref{eq:formula_eta}. For all measurable $f\geq 0$ and all $0\leq a<x<b\leq \infty$,
\begin{align}
\label{eq:green-function}
\E_x\left(\int_{0}^{T_a\wedge T_b}f(X_t)\,dt\right)
=2\int_a^b \frac{(x\wedge y-a)(b-x\vee y)}{b-a}\,f(y)\,m(dy).
\end{align}
For $a=0$ and $b=\infty$, we obtain
\begin{align}
\label{eq:green-function_0_infty}
\E_x\left(\int_{0}^{\tau_\d}f(X_t)\,dt\right)
=2\int_0^\infty (x\wedge y)\,f(y)\,m(dy).
\end{align}
For $f=\eta$, we deduce that
\begin{align*}
\int_0^\infty \E_x\left(\eta(X_t)\11_{t<\tau_\d}\right)\,dt
&=2\int_0^\infty (x\wedge y)\,\eta(y)\,m(dy).
\end{align*}
Since $\eta(0)=0$ and $L\eta=-\lambda_0\eta$, we have $\E_x\left(\eta(X_t)\11_{t<\tau_\d}\right)=P_t\eta(x)=e^{-\lambda_0 t}\eta(x)$. Then
\begin{align*}
\frac{\eta(x)}{\lambda_0}=2\int_0^\infty (x\wedge y)\,\eta(y)\,m(dy).
\end{align*}
This entails~\eqref{eq:formula_eta} provided we prove that $\alpha(dx)$ is proportional to $\eta(x)m(dx)$ (observe that the
normalizing constant is determined by the condition $\alpha(\eta)=1$). This will be done in the next Subsection.

\subsection{Proof of~\eqref{eq:expression-of-alpha} and that $\int y\,\alpha(dy)<\infty$}
\label{sec:proof-second-part}

Integrating~\eqref{eq:green-function_0_infty} with respect to $\alpha(dx)$, we obtain
\begin{align*}
\int_0^\infty \E_{\alpha}\left(f(X_t)\11_{t<\tau_\d}\right)dt=2\int_0^\infty\int_0^\infty (x\wedge y)\,f(y)\,m(dy)\,\alpha(dx).
\end{align*}
Since $\E_{\alpha}\left(f(X_t)\11_{t<\tau_\d}\right)=\alpha(f)\,e^{-\lambda_0 t}$, we deduce that
\begin{align*}
\frac{\alpha(f)}{\lambda_0}=2\int_0^\infty\int_0^\infty (x\wedge y)\,f(y)\,m(dy)\,\alpha(dx).
\end{align*}
This entails~\eqref{eq:expression-of-alpha}. We now prove that $\alpha(dx)$ is proportional to $\eta(x)m(dx)$. 

The following reversibility result for diffusions on natural scale is more or less classical but we need a version with precise
bounds on the test functions in the case of a diffusion coming down from infinity and hitting 0 a.s.\ in finite time. The proof is
given at the end of the subsection for sake of completeness.

\begin{lem}
  \label{lem:reversibility}
  Let $X$ be a diffusion on $[0,\infty)$ in natural scale with speed measure $m$ satisfying~\eqref{eq:ymdy_integrable}. Then it is
  reversible with respect to $m$ in the sense that, for all bounded non-negative measurable functions $f$ on $(0,+\infty)$ and all
  nonnegative measurable function $g$ on $(0,+\infty)$ such that $g(x)\leq Cx$ for some $C>0$,
  \begin{equation}
    \label{eq:reversibility}
    \int_0^\infty f(x)P_tg(x)m(dx)=\int_0^\infty g(x)P_tf(x) m(dx),\quad\forall t\geq 0,
  \end{equation}
  where both sides are finite.
\end{lem}

Applying this lemma for $g=\eta$, we obtain that, for all bounded measurable non-negative $f$,
$$
\int_0^\infty f(x)\eta(x)m(dx)=e^{\lambda_0 t}\int_0^\infty \eta(x)P_tf(x)m(dx)
$$
Now, it follows from~\eqref{eq:expo-cv} and~\eqref{eq:convergence-to-eta} that
$$
e^{\lambda_0 t}P_t f(x)=e^{\lambda_0 t}\PP_x(t<\tau_\d)\EE_x(f(X_t)\mid t<\tau_\d)\rightarrow \eta(x) \alpha(f)
$$
when $t\rightarrow+\infty$, where the convergence is uniform in $x$. Since $\int\eta\,dm<\infty$, Lebesgue's theorem entails that
$$
\alpha(f)\int_0^\infty\eta^2(x) m(dx)=\int_0^\infty f(x)\eta(x)m(dx)
$$
for all bounded measurable $f$. Hence $\alpha\propto \eta\,dm$. Note also that $f=\eta$ gives $\int_0^\infty \eta\,dm=\int_0^\infty
\eta^2\,dm$.

It only remains to prove that $\int y\,\alpha(dy)<\infty$. Since $\eta$ is bounded and
$$
\eta(x)\propto\int_0^\infty (x\wedge y)\,\alpha(dy),
$$
this follows from the monotone convergence theorem.

\begin{proof}[Proof of Lemma~\ref{lem:reversibility}]
We start by proving that both sides of~\eqref{eq:reversibility} are finite. This is obvious for the r.h.s.\ because
of~\eqref{eq:ymdy_integrable}. For the l.h.s., since the positive local martingale $(X_t,t\geq 0)$ is a supermartingale, we have
$P_tg(x)=\E_x[g(X_t)\11_{t<\tau_\d}]\leq C\E_x[X_t]\leq Cx$, which allows to conclude.

By Lebesgue's theorem, it is enough to prove~\eqref{eq:reversibility} for a $f,g$ in a dense subset of the set of continuous
functions on $(0,+\infty)$ with compact support. Note that a function $g$ with compact support in $(0,+\infty)$ satisfies $g(x)\leq C
x$ for some $C>0$. For all $s\in[0,t]$, we define
$$
\psi(s)=\int_0^\infty P_sf(x)\,P_{t-s}g(x)\,m(dx).
$$

We use the characterization of the infinitesimal generator of diffusion processes of~\cite[Thm.\,2.81]{freedman-83}: let $\mathcal{D}$
be the set of functions $f$ bounded continuous on $[0,\infty)$, such that the right derivative $f^+$ of $f$ exists, is finite,
continuous from the right and of bounded variation on all compact intervals of $(0,\infty)$, and such that $df^+=h\,dm$, where $df^+$
denotes the measure on $(0,\infty)$ such that
$$
f^+(y)-f^+(x)=df^+(x,y]
$$
and $h$ is some bounded continuous function on $[0,\infty)$ with $h(0)=0$. Then, for all $f\in\mathcal{D}$, 
$$
Lf(x):=\lim_{t\rightarrow 0}\frac{P_tf(x)-f(x)}{t}=\frac{df^+}{dm}(x),\quad\forall x\geq 0,
$$
where the convergence holds for the uniform norm on $[0,\infty)$.

So let $f,g\in \mathcal{D}$ have compact support. In particular, there exists $C>0$ such that $P_sf(x)+P_sg(x)\leq Cx$ for all
$x,s\geq 0$. Since $f,g\in\mathcal{D}$, $P_sf$ and $P_sg$ also belong to $\mathcal{D}$ and Lebesgue's theorem entails that, for all
$s\in[0,t]$
\begin{align*}
  \psi'(s) & =\int_0^\infty LP_s f(x)\,P_{t-s}g(x)\,m(dx)-\int_0^\infty P_s f(x)\,LP_{t-s}g(x)\,m(dx) \\
  & =\int_0^\infty P_{t-s}g(x)\,d(LP_s f)^+(x)-\int_0^\infty P_s f(x)\,d(LP_{t-s}g)^+(x).
\end{align*}
The right-hand side is equal to zero by the integration by parts formula for Stieljes integrals.

Since $\mathcal{D}$ is dense in the set of continuous functions on $(0,\infty)$ with compact support, the proof is completed.
\end{proof}

\subsection{Proof of Proposition~\ref{prop:jump}}
\label{sec:proof-prop-jump}

The proof is similar to the one of Theorem~\ref{thm:QSD_full} (see Subsection~\ref{sec:proof-first-part}) and we only detail the
steps that need to be modified. Our aim is to check that conditions (A1) and (A2) hold.

\bigskip\noindent
\textit{Step 1. (A1) is satisfied.}

Since $m$ satisfies the conditions of Theorem~\ref{thm:QSD_full}, we deduce that (A1) holds for $X$. As a consequence, there exist
two constants $c_1>0$ and $t_0>0$ and a probability measure $\nu$ on $(0,+\infty)$, such that, for all $A\subset(0,\infty)$ measurable,
\begin{align*}
\P_x(X_{t_0}\in A)\geq c_1\nu(A)\P_x(t_0<\tau_\d),\quad \forall x\in(0,+\infty).
\end{align*}
By construction of $\tX$, we have
\begin{align*}
\P_x(\tX_{t_0}\in A)\geq e^{-t_0} \P_x(X_{t_0}\in A).
\end{align*}
Fix $x\in(0,1)$. Using the fact that $X_t=\tX_t$ for all $t\leq T_1$ under $\P_x$, we deduce that
\begin{align*}
\P_x(t_0<\widetilde{\tau_\d})&\leq \P_x(\{t_0<\tau_\d\}\cup\{\tX_{t_0}\neq X_{t_0}\})\\
&\leq \P_x(t_0<\tau_\d)+\P_x(T_1\leq \tau_\d)\\
&\leq Ax+x\\
&\leq \frac{A+1}{a}\P_x(t_0<\tau_\d),
\end{align*}
where $a>0$ is the positive constant from Remark~\ref{rem:cond-(B)}. As a consequence, for all $A\subset(0,\infty)$
measurable,
\begin{align}
\label{eq:comp-tX-X}
\P_x(\tX_{t_0}\in A)\geq \frac{c_1 a e^{-t_0}}{A+1}\,\nu(A)\,\P_x(t_0<\widetilde{\tau_\d}),
\end{align}
which concludes the proof of~(A1) for $\widetilde{X}$.

\bigskip\noindent \textit{Step~2. (A2) is satisfied.}

By construction of the process $\tX$ from the process $X$ and an independent Poisson process, it is clear that, for all $x\leq y$,
all $t>0$ and $a\in(0,x)$, we have
\begin{align*}
  \P_x(\tT_a\leq t)\geq \P_y(\tT_a\leq t),
\end{align*}
where $\tT_a=\inf\{t\geq 0,\ \tX_t=a\}$, and that $\P_x(\tX_t\geq a)\geq \P_x(X_t\geq a)$. Hence we only need to prove that $\tX$
satisfies~\eqref{eq:moment-expo}, the rest of the proof being the same as in Step~4 of Subsection~\ref{sec:proof-first-part}.

Fix $\varepsilon>0$ and set $t_\varepsilon=-\log(1-\varepsilon/2)>0$. For all $a>0$, we have, by independence of the Poisson process
$(N_t)_{t\geq 0}$,
\begin{align*}
  \inf_{x\in(a,+\infty)} \P_x(\tT_a\leq t_\varepsilon) =\lim_{x\rightarrow\infty}\P_x(\tT_a\leq t_\varepsilon)\geq
  \left(1-\frac{\varepsilon}{2}\right)\lim_{x\rightarrow\infty}\P_x(T_a\leq t_\varepsilon).
\end{align*}
Since $X$ comes down from infinity, there exists $a_\varepsilon>0$ such that
$$
\lim_{x\rightarrow+\infty}\P_x(T_{a_\varepsilon}\leq t_\varepsilon)\geq 1-\frac{\varepsilon}{2}.
$$
Hence, for all $\varepsilon>0$ small enough, there exists $a_\varepsilon>0$ such that, for all $x\geq a_\varepsilon$,
\begin{align*}
  \P_x(1<\tT_{a_\varepsilon})\leq \P_x(t_\varepsilon<\tT_{a_\varepsilon})\leq 1-\left(1-\frac{\varepsilon}{2}\right)^2\leq\varepsilon.
\end{align*}
This entails~\eqref{eq:moment-expo} and (A2) as in Step~4 of the proof of Theorem~\ref{thm:QSD_full}.

\bibliographystyle{abbrv}
\bibliography{biblio-bio,biblio-denis,biblio-math,biblio-math-nicolas}

\def\cprime{$'$} \def\cprime{$'$} \def\cprime{$'$} \def\cprime{$'$}
  \def\polhk#1{\setbox0=\hbox{#1}{\ooalign{\hidewidth
  \lower1.5ex\hbox{`}\hidewidth\crcr\unhbox0}}} \def\cprime{$'$}
  \def\lfhook#1{\setbox0=\hbox{#1}{\ooalign{\hidewidth
  \lower1.5ex\hbox{'}\hidewidth\crcr\unhbox0}}} \def\cprime{$'$}
  \def\cprime{$'$}
\begin{thebibliography}{10}

\bibitem{amir-91}
M.~Amir.
\newblock Sticky brownian motion as the strong limit of a sequence of random
  walks.
\newblock {\em Stochastic Processes and their Applications}, 39(2):221 -- 237,
  1991.

\bibitem{Cattiaux2009}
P.~Cattiaux, P.~Collet, A.~Lambert, S.~Martinez, S.~M{\'e}l{\'e}ard, and
  J.~San~Martin.
\newblock Quasi-stationary distributions and diffusion models in population
  dynamics.
\newblock {\em Ann. Probab.}, 37(5):1926--1969, 2009.

\bibitem{ChampagnatVillemonais2016}
N.~Champagnat and D.~Villemonais.
\newblock Exponential convergence to quasi-stationary distribution and
  {Q}-process.
\newblock {\em Probability Theory and Related Fields}, 164(1):243--283, 2016.

\bibitem{champagnat-villemonais-16c}
N.~Champagnat and D.~Villemonais.
\newblock Uniform convergence of penalized time-inhomogeneous markov processes.
\newblock Arxiv preprint, 2016.

\bibitem{ChampagnatVillemonais2016U}
N.~{Champagnat} and D.~{Villemonais}.
\newblock {Uniform convergence to the Q-process}.
\newblock {\em ArXiv e-prints}, Nov. 2016.

\bibitem{freedman-83}
D.~Freedman.
\newblock {\em Brownian motion and diffusion}.
\newblock Springer-Verlag, New York-Berlin, second edition, 1983.

\bibitem{kolb-hening-2104}
A.~{Hening} and M.~{Kolb}.
\newblock {Quasistationary distributions for one-dimensional diffusions with
  singular boundary points}.
\newblock {\em ArXiv e-prints}, Sept. 2014.

\bibitem{ito-mckean-74}
K.~It{\^o} and H.~P. McKean, Jr.
\newblock {\em Diffusion processes and their sample paths}, volume 125 of {\em
  Die Grundlehren der mathematischen Wissenschaften}.
\newblock Springer-Verlag, Berlin, 1974.
\newblock Second printing, corrected.

\bibitem{kallenberg-02}
O.~Kallenberg.
\newblock {\em Foundations of modern probability}.
\newblock Probability and its Applications (New York). Springer-Verlag, New
  York, second edition, 2002.

\bibitem{kolb-steinsaltz-12}
M.~Kolb and D.~Steinsaltz.
\newblock Quasilimiting behavior for one-dimensional diffusions with killing.
\newblock {\em Ann. Probab.}, 40(1):162--212, 2012.

\bibitem{kitani-06}
S.~Kotani.
\newblock On a condition that one-dimensional diffusion processes are
  martingales.
\newblock In {\em In memoriam {P}aul-{A}ndr\'e {M}eyer: {S}\'eminaire de
  {P}robabilit\'es {XXXIX}}, volume 1874 of {\em Lecture Notes in Math.}, pages
  149--156. Springer, Berlin, 2006.

\bibitem{legall-83}
J.-F. Le~Gall.
\newblock Applications du temps local aux \'equations diff\'erentielles
  stochastiques unidimensionnelles.
\newblock In {\em Seminar on probability, {XVII}}, volume 986 of {\em Lecture
  Notes in Math.}, pages 15--31. Springer, Berlin, 1983.

\bibitem{littin-12}
J.~Littin~C.
\newblock Uniqueness of quasistationary distributions and discrete spectra when
  {$\infty$} is an entrance boundary and 0 is singular.
\newblock {\em J. Appl. Probab.}, 49(3):719--730, 2012.

\bibitem{matsumoto-89}
H.~Matsumoto.
\newblock Coalescing stochastic flows on the real line.
\newblock {\em Osaka J. Math.}, 26(1):139--158, 1989.

\bibitem{meleard-villemonais-12}
S.~M{\'e}l{\'e}ard and D.~Villemonais.
\newblock Quasi-stationary distributions and population processes.
\newblock {\em Probab. Surv.}, 9:340--410, 2012.

\bibitem{Meyn2009}
S.~Meyn and R.~Tweedie.
\newblock {\em Markov chains and stochastic stability}.
\newblock Cambridge University Press New York, NY, USA, 2009.

\bibitem{miura-14}
Y.~Miura.
\newblock Ultracontractivity for {M}arkov semigroups and quasi-stationary
  distributions.
\newblock {\em Stoch. Anal. Appl.}, 32(4):591--601, 2014.

\bibitem{petrowski-ruf-2012}
N.~Perkowski and J.~Ruf.
\newblock Conditioned martingales.
\newblock {\em Electron. Commun. Probab.}, 17:no. 48, 12, 2012.

\bibitem{pollett-11}
P.~K. Pollett.
\newblock Quasi-stationary distributions: A bibliography.
\newblock available at http://www.maths.uq.edu.au/{\verb+~+}pkp/papers/qsds/,
  2011.

\bibitem{Protter2013}
P.~Protter.
\newblock A mathematical theory of financial bubbles.
\newblock In {\em Paris-Princeton Lectures on Mathematical Finance 2013},
  volume 2081 of {\em Lecture Notes in Mathematics}, pages 1--108. Springer
  International Publishing, 2013.

\bibitem{Revuz1999}
D.~Revuz and M.~Yor.
\newblock {\em Continuous martingales and {B}rownian motion}, volume 293 of
  {\em Grundlehren der Mathematischen Wissenschaften}.
\newblock Springer-Verlag, Berlin, third edition, 1999.

\bibitem{vanDoorn2013}
E.~A. van Doorn and P.~K. Pollett.
\newblock Quasi-stationary distributions for discrete-state models.
\newblock {\em European J. Oper. Res.}, 230(1):1--14, 2013.

\end{thebibliography}

\end{document}